\documentclass[12pt]{amsart}
\usepackage{amscd, amssymb}
\usepackage{epic}
\usepackage{eepic}
\usepackage{epsfig}
\newtheorem{Thm}{Theorem}[section]

\newtheorem{Prop}[Thm]{Proposition}
\newtheorem{Def}[Thm]{Definition}
\newtheorem{Def/Thm}[Thm]{Definition/Theorem}
\newtheorem{Cor}[Thm]{Corollary}
\newtheorem{Lemma}[Thm]{Lemma}

\theoremstyle{remark}
\newtheorem{Rmk}[Thm]{Remark}

\numberwithin{equation}{subsection}

\newcommand{\cN}{{\mathcal{N}}}
\newcommand{\cA}{{\mathcal{A}}}

\newcommand{\cI}{{\mathcal{I}}}
\newcommand{\cT}{{\mathcal{T}}}

\newcommand{\cS}{{\mathcal{S}}}
\newcommand{\cG}{{\mathcal{G}}}
\newcommand{\cCH}{{\mathcal{CH}}}
\newcommand{\cJ}{{\mathcal{J}}}

\begin{document}

%\title{The Abelian/Nonabelian Correspondence and Frobenius Manifolds}
%\author{Ionu\c t Ciocan-Fontanine}
%\address{School of Mathematics, University of Minnesota,
%Minneapolis MN, 55455, USA} \email{ciocan@math.umn.edu}
%
%\author{Bumsig Kim}
%\address{School of Mathematics, Korea Institute for Advanced Study,
%207-43 Cheongnyangni 2-dong, Dongdaemun-gu, Seoul, 130-722, Korea}
%\email{bumsig@kias.re.kr}
%
%\author{Claude Sabbah}
%\address{UMR 7640 du C.N.R.S., Centre de math\'ematiques Laurent Schwartz,
%\'Ecole polytechnique, F-91128 Palaiseau cedex, France}
%\email{sabbah@math.polytechnique.fr}
%
%\maketitle

\title{The Chow ring of relative Fulton--MacPherson space}

\author{fumitoshi Sato}
\address{School of Mathematics\\ Korean Institute for Advanced Study}
\email{fumi@kias.re.kr}

\bibliographystyle{amsplain}

%%%%%%%%%%%%
% Abstract %
%%%%%%%%%%%%

\begin{abstract}
Suppose that $X$ is a nonsingular variety and $D$ is a nonsingular proper subvariety.  
   Configuration spaces of distinct and non-distinct $n$ points in $X$ away from $D$ were constructed by the author and B. Kim in \cite{KS} by using the method of wonderful compactification.  In this paper, we give an explicit presentation of Chow motives and Chow rings of these configuration spaces.
\end{abstract}

\maketitle

%%%%%%%%%%%%%%%%%%%%%%%%%%%%%%%%%%%%%%%%%%%%%%%%%%%%%%

%%%%%%%%%%%%%%%%
% Introduction %
%%%%%%%%%%%%%%%%

\section{Introduction}
Let $X$ be a complex connected nonsingular algebraic variety and let $D$ be a smooth divisor.
  
In \cite{KS}, two generalizations of Fulton--MacPherson spaces were constructed by using the method of wonderful compactifications \cite{Li1}.  Two spaces are following:
\begin{enumerate}
\item A compactification $X^{[n]}_D$ of the configuration space of $n$ labeled points in $X \setminus D$, i.e. "not allowing those points to meets $D$."

\item A compactification $X_D [n]$ of the configuration spaces of $n$ distinct labeled points in $X \setminus D$, i.e. "not allowing those points to meet each other as well as $D$."
\end{enumerate}   

The goal of this paper is to give an explicit presentation of Chow motives and Chow rings of these configuration spaces.  Our main theorems are:

 \begin{Thm}
The Chow ring $A^*(X_D^{ [n]})$ is isomorphic to the polynomial
ring $ A^*(X^n) [x_S]$ modulo the ideal generated by

\begin{enumerate}

\item $x_{S} \cdot x_{T}$ for  $ S, T $ that overlap,

\item $\cJ_{D_{S}/X^n} \cdot x_{S}$ for all $S$,

\item $P_{D_{S}/X^n} (-\Sigma_{S' \supset S} x_{S'})$ for
all $S$.
\end{enumerate}

\end{Thm}

\begin{Thm}
The Chow ring $A^*(X_D [n])$ is isomorphic to the polynomial
ring $ A^*(X^n) [x_S, y_I] $ modulo the ideal generated by

\begin{enumerate}
\item $y_{I} \cdot y_{J}$ for $I$ and $J$ that overlap,

\item $x_{S} \cdot x_{T}$ for $S$ and $T$ that overlap,

\item $x_{S} \cdot y_{I}$ unless $I \subset S$,

\item $\cJ_{\Delta_{I}/X^n} \cdot y_{I}$ for all $I$,

\item $\cJ_{D_{S}/X^n} \cdot x_{S}$ for all $S$,

\item $c_{a,b} ( \sum_{a, b \in I} y_I)$ for $a, b \in \{1, \cdots, n \}$ (distinct),

\item $P_{D_{S}/X^n} (-\Sigma_{S' \supset S} x_{S'})$ for
all $S$.

\end{enumerate}

\end{Thm}

The paper is organized as follows.  In section \ref{wonderful}, we review theory of wonderful compactification and Chow rings and motives after blow-up.  In section \ref{construction}, we review the construction of compactifications of $n$ points in $X \setminus D$. In section \ref{groups-motives}, we compute Chow groups and motives explicitly.  In section \ref{rings}, we compute Chow rings under the assumptions such that $X^n$ has the Kunneth decomposition and the embedding $D \hookrightarrow X$ is a Lefshetz embedding.
   
\subsection{Notation}\label{notation}
\begin{itemize}
\item As in \cite{FM}, for
 a subset $I$ of $N:=\{1,2,...,n\}$, let \[ I^+:=I \cup \{ n+1
 \}. \]

\item Let $Y_1$ be the blowup of a nonsingular complex variety $Y_0$
along a nonsingular closed subvariety $Z$. If $V$ is an
irreducible subvariety of $Y_0$, we will use $\widetilde{V}$ or
$V(Y_1)$  to denote

\begin{itemize}
\item the total transform of $V$, if $V \subset Z$;

\item the proper transform of $V$, otherwise.
\end{itemize}

If there is no risk to cause confusion, we will use simply $V$ to
denote $\widetilde{V}$. The space $\mathrm{Bl}_{\widetilde{V}}Y_1$
will be called the iterated blowup of $Y_0$ along centers $Z, V$
(with the order).

\item For a partition of $I$ of $N$, $\Delta _I$ denotes the
polydiagonal associated to $I$. And consider the binary operation
$I\wedge J$ on the set of all partitions satisfying \[ \Delta
_I\cap \Delta _J = \Delta _{I\wedge J}. \] We use $\Delta _{I_0}$ instead of $\Delta _I$ when $I=\{
I_0, I_1,...,I_l \}$ such that $|I_i|=1$ for all $i\ge 1$.

\end{itemize}

\subsection{Acknowledgements} The author thanks Bumsig Lim, Li Li, Philipo Viviani, and Stephanie Yang for useful discussions.  Most of the work took place at Mittag-Leffler Institute, Sweden while he was attending the program "Moduli Spaces" and the author thanks for its hospitality. 
   
\section{Wonderful Compactification of Arrangements of
Subvarieties}
\label{wonderful}

In this section, we review the theory of wonderful
compactification of arrangements of subvarieties.  See the detail
and proofs in \cite{Li1}, \cite{Li2}.

\subsection{Arrangement, building set and nest}

\begin{Def} [of clean intersection]
Let $Y$ be a nonsingular algebraic variety and let $U$ and $V$ be
two smooth subvarieties of $Y$.

$U$ and $V$ intersect cleanly if $U \neq V $ and their
scheme-theoretic intersection is nonsingular and the tangent
bundles satisfy $T(U \cap V) = TU \cap TV$.
\end{Def}

\begin{Rmk}
If the intersection is transversal, then it is a clean
intersection.
\end{Rmk}

\begin{Def}[of arrangement]
A simple arrangement of subvarieties of $Y$ is a finite set $\cS= \{ S_i
\}$ of nonsingular closed irreducible subvarieties of $Y$
satisfying the following conditions

\begin{enumerate}
\item $S_i$ and $S_j$ intersect cleanly,

\item $S_i \cap S_j$ is either empty or some
$S_k$'s.
\end{enumerate}

\end{Def}

\begin{Def}[of building set]
Let $\cS$ be an arrangement of subvarieties of $Y$.  A subset $\cG
\subset \cS$ is called a building set with respect to $\cS$, if ,
for any $S \in \cS$, the minimal elements in $\cG$ which contain
$S$ intersect transversally and their intersection is $S$.  These
minimal elements are called the $\cG$-factors of $S$.
\end{Def}

\begin{Def}[of $\cG$-nest]
A subset $\cT \subset \cG$ is called a $\cG$-nest if there is a
flag of elements in $\cS$; $S_1 \subset S_2 \subset \cdots \subset
S_k$ such that
$$ \cT = \cup_{i=1}^{k} \{ A: \text{$A$ is a $\cG$-factor of $S_i$ } \}.$$
\end{Def}

\subsection{Construction of $Y_{\cG}$ by a sequence of blow-ups}

Let $Y$ be a nonsingular algebraic variety, $\cS$ be a simple
arrangement of subvarieties and $\cG$ be a building set with
respect to $\cS$. Order $\cG= \{ G_1, \cdots, G_N\}$ such that $i<
j$ if $G_i \subset G_j$.

We define $(Y_k, \cS^{(k)}, \cG^{(k)})$ inductively, where $Y_k$
is a blow-up of $Y_{k-1}$ along a nonsingular variety, $\cS^{(k)}$
is a simple arrangement of subvarieties of $Y_k$ and $\cG^{(k)}$
is a building set with respect to $\cS^{(k)}$.

\begin{Def/Thm}
\label{thm:one-blow} Assume $\cS$ is a simple arrangement of
subvarieties of $Y$ and $\cG$ is a building set.  Let $G$ be a
minimal element in $\cG$ and consider $\pi:
\widetilde{Y}:=\mathrm{Bl}_{G}{Y} \rightarrow Y$. Denote the
exceptional divisor by $E$. For any nonsingular variety $V$ in
$Y$, we define $\widetilde{V} \subset \mathrm{Bl}_{G}{Y}$, the
$\sim$ transform of $V$, to be the proper transform of $V$ if $V
\nsubseteq G$, and to be $\pi^{-1}(V)$ if $V \subset G$.

For simplicity of notation, for a sequence of blow-ups, we use the
same notation $\widetilde{V}$ to denote the iterated one.

\begin{enumerate}
\item The collection $\cS'$ of subvarieties in $\widetilde{Y}$
defined by
$$\cS' := \{ \widetilde{S} \}_{S \in \cS} \cup \{ \widetilde{S} \cap
E \}_{\emptyset \varsubsetneq S \cap G \varsubsetneq S}$$ is a
simple arrangement in $\tilde{Y}$

\item $\cG':= \{ \widetilde{G_i} \}_{G_i \in \cG}$ is a building
set with respect to $\cS'$.

\item Given a subset $\cT$ of $\cG$.  Define $\cT' := \{
\widetilde{A} \}_{A \in \cT} \subset \cG'$.  $\cT $ is a
$\cG$-nest if and only if $\cT'$ is a $\cG'$-nest.
\end{enumerate}
\end{Def/Thm}

Let's go back to the construction of $Y_\cG$.
\begin{enumerate}
\item For $k=0$, $Y_0=Y, \cS^{(0)}= \cS, \cG^{(0)}=\cG= \{G_1,
\cdots, G_N \}, G_i^{(0)}=G_i $.

\item Assume $Y_{k-1}$ is already constructed.  Let $Y_k$ be the
blow-up of $Y_{k-1}$ along the nonsingular subvariety
$G_k^{(k-1)}$. Define $G_{i}^{(k)} := \widetilde{ G_{i}^{(k-1)}
}$. Since $G_{i}^{(k-1)}$ for $i<k$ are all divisors,
$G_{k}^{(k-1)}$ is minimal in $\cG^{(k-1)}$.  Thus there is a
naturally induced arrangement $\cS^{(k)}$ and a building set
$\cG^{(k)}$ by the theorem \ref{thm:one-blow}.

\item Continue the inductive construction to $k=N$, where all
elements in the building set $\cG^{(N)}$ are divisors.
\end{enumerate}

\begin{Thm}
\label{thm:existance}
Denote $Y^{\circ}=Y \setminus \cup_{G \in
\cG} G$. There is a natural locally closed embedding
$$  Y^{\circ} \hookrightarrow Y \times \prod_{G \in \cG} \mathrm{Bl}_{G}Y,$$
and its closure is denoted by $Y_{\cG}$ and called the wonderful
compactification of $\cG$.  Then $Y_{\cG}$ is isomorphic to $Y_N$
which is constructed in the above.

The variety $Y_{\cG}$ is nonsingular.  For each $G \in \cG$, there
is a nonsingular divisors $D_G \subset Y_{\cG}$ such that

\begin{enumerate}
\item The union of these divisors is $Y_{\cG} \setminus
Y^{\circ}.$

\item Any set of these divisors meets transversally.  An
intersection of divisors $D_{T_1} \cap \cdots D_{T_l}$ is not
empty exactly when $\{ T_1, \cdots T_l \}$ form a $\cG$-nest.
\end{enumerate}
\end{Thm}

\begin{Thm}[order of blow-ups]
\label{thm:order}
\begin{enumerate}
\item Let $\cI_i$ be the ideal sheaf of $G_i \in \cG.$  Then
$$ Y_{\cG} \cong \mathrm{Bl}_{\cI_1 \cdots \cI_N} Y. $$

\item If we arrange $\cG = \{ G_1, \cdots G_N \}$ in such an order
that
$$ \text{ $(*)$ for any $1 \leq i \leq N$, the first $i$ terms $G_1,
\cdots G_i$ form a building set}$$

Then

$$ Y_{\cG} \cong \mathrm{Bl}_{\widetilde{G_N}} \cdots \mathrm{Bl}_{\widetilde{G_2}}
\mathrm{Bl}_{G_1}Y,$$

where each blow-up is along a smooth subvariety.
\end{enumerate}

\end{Thm}

\subsection{Chow  group and motive of $Y_{\cG}$ }

 Let $Y_0 :=Y, Y_0 \cT :=\cap_{T \in \cT} T$ where $\cT$
is a $\cG$-nest. Define $r_{\cT}(G) := \text{dim} (\cap_{G
\subsetneq T \in \cT} T) - \text{dim} G$ (here we use a convention
that $\cap_{G \subsetneq T \in \cT} T =Y$  if no $T $ strictly
contains $G$). Then define
$$ M_{\cT} := \{ \overrightarrow{\mu}= \{ \mu_G \}_{G \in \cT} :1
\leq \mu_G \leq r_{\cT} (G) -1 \}$$ and let $\|
\overrightarrow{\mu} \| := \sum_{G \in \cG} \mu_G$ for
$\overrightarrow{\mu} \in M_{\cT}$.

\begin{Thm}
\label{ChowY}
We have the Chow group decomposition
$$ A^*( Y_{\cG}) = A^*(Y) \oplus \bigoplus_{\cT} \bigoplus_{  \overrightarrow{\mu} \in M_{\cT}
} A^{*- \| \overrightarrow{\mu} \| } (Y_0 \cT)$$ where $\cT$ runs
through all $\cG$-nests.

If $Y$ is complete, we also have the Chow motive decomposition
$$ h( Y_{\cG}) = h(Y) \oplus \bigoplus_{\cT} \bigoplus_{\overrightarrow{\mu} \in M_{\cT}
}  h(Y_0 \cT)( \| \overrightarrow{\mu} \|)$$ where $\cT$ runs
through all $\cG$-nests.

\end{Thm}

\subsection{Chow ring of $Y_{\cS}$}

In this section, we will review the result of Chow rings after blow-up \cite{Keel} and the result of Hu \cite{Hu}
concerning the Chow ring of $Y_{\cS}$.

\begin{Def}[of Lefschetz embedding]
An embedding $U \hookrightarrow Y$ is called a Lefshetz embedding
if the restriction map $A^* (Y) \rightarrow A^*(U)$ is surjective.
Under this situation, let $\cJ_{U/Y}$ be the kernel of $A^* (Y)
\rightarrow A^*(U)$ and let $P_{U/Y}$ be the Chern polynomial for
the normal bundle $N_{U/Y}$.
\end{Def}

\begin{Def}
A Chern polynomial $P_{U/Y}(t)$ for a Lefshetz embedding $U \hookrightarrow Y$ is a polynomial
$$P_{U/Y}(t) = t^d +a_1 t^{d-1}+ \cdots a_{d-1}t + a_d \in A^*(Y)[t],$$ 
where $d$ is the codimension of $U$ in Y and $a_i \in A^i(Y)$ is a class whose restriction in $A^i(U)$ is the Chern class $c_i(N_{U/Y})$.
\end{Def}

\begin{Lemma}
\begin{enumerate}

\item If $D$ is a divisor, then $P_{D/Y}(t) = t+D$,

\item If $V_1, V_2 \subset Y$ are subvarieties meeting transversally and their intersection is $Z$, then 
$$ P_{Z/Y}(t) = P_{V_1/Y}(t) \cdot P_{V_2/Y}(t)$$

\end{enumerate}
\end{Lemma}

\begin{Lemma}\label{lefker}
Let $U$ and $V$ are non-singular closed subvarieties of $Y$ meeting cleanly in a non-singular closed subvariety $Z$.  We also assume that both embeddings $U \hookrightarrow Y$ and $V \hookrightarrow Y$ are Lefschetz.  Then all the relevant inclusions below are Lefschetz and 

\begin{enumerate}
\item $\cJ_{\mathrm{Bl}_Z U/ \mathrm{Bl}_V Y} = \cJ_{U/Y}$ if 
$Z \neq \emptyset$,

\item $\cJ_{\mathrm{Bl}_Z U/ \mathrm{Bl}_V Y}= (\cJ_{U/Y}, \widetilde{V}) $ if $Z = \emptyset$, where $\widetilde{V}$ is the exceptional divisor in $\mathrm{Bl}_V Y$,

\item $\cJ_{\mathrm{Bl}_Z U/ \mathrm{Bl}_Z Y} = (\cJ_{U/Y}, [ \mathrm{Bl}_Z V ])$ if $Z \neq \emptyset$,

\item $P_{\mathrm{Bl}_Z U/ \mathrm{Bl}_V Y}(t) = P_{U/Y}(t)$,

\item $P_{\mathrm{Bl}_Z U/ \mathrm{Bl}_Z Y}(t) = P_{U/Y}(t- \widetilde{Z})$ where $\widetilde{Z}$ is the exceptional divisor in $\mathrm{Bl}_Z Y$.
\end{enumerate}  
\end{Lemma}

\begin{Lemma}\label{blowchow}
Let $\{ U_i \}$ be disjoint non-singular closed subvarieties of a smooth variety $Y$, such that $ U_i \hookrightarrow Y$ are Lefschetz.  Then the Chow ring $A^*( \mathrm{Bl}_{ \cup U_i} Y)$ is isomorphic the polynomial ring $A^*(Y)[x_i]$
,where $x_i$ corresponds to the exceptional divisor $\widetilde{U_i}$, modulo the ideal generated by 

\begin{enumerate}
\item $x_i \cdot x_j$ for $i \neq j$,

\item $\cJ_{U_i/ Y} \cdot x_i$ for all $i$,  

\item $P_{U_i / Y}( -x_i)$ for all $i$.
\end{enumerate}

\end{Lemma}

\begin{Def}
A regular simple arrangement $\cS$ is a simple arrangement such
that for any $S_l \subset S_i$, there is $S_j \supset S_l$ such
that $S_l = S_i \cap S_j$.
\end{Def}

\begin{Thm}\label{chowring}
Let $\cS$ be a regular simple arrangement of subvarieties such
that all the inclusions $S_i \subset S_j$ and $S_i \subset Y$ are
Lefschetz embedding.  Then the Chow ring of $Y_{\cS}$ is
isomorphic to the polynomial ring $ A^*(Y)  [ x_{S_{1}}, \cdots, x_{S_{N}} ]      $ (where $x_{S_i}$ corresponds to the exceptional divisor $S_{i}^{(i+1)}$) 
modulo the ideal generated by

\begin{enumerate}

\item $x_{S_i} \cdot x_{S_j}$ for incomparable $S_i, S_j $,

\item $\cJ_{S_i/Y} \cdot x_{S_i}$ for all $i$,

\item $P_{S_i/Y} (- \Sigma_{S_j \subseteq S_i} x_{S_j})$ for all $i$.

\end{enumerate}

\end{Thm}

\section{Construction of $ X_D^{[ n ] } $ and $ X_D [ n ]  $}
\label{construction}

Fix a nonsingular divisor $D$ of an algebraic variety $X$ of
dimension $m$.  In this section, we review constructions of  a
compactification of configuration spaces of $n$ point in $X
\setminus D$, $ X_D^{[ n ] }$, and a compactification of
configuration spaces of $n$ distinct point in $X \setminus D$, $
X_D{[ n ] }$.  In this paper, we assume that $D$ is a divisor but every thing will work in the case of $D$ is a smooth subvariety after some adjustment.  See the details in  \cite{KS}.

\subsection{Construction}
For a subset $S$ of $N:=\{1,2,...,n\}$ define a nonsingular
subvariety in $X^n$
$$D_S := \{{\bf x}\in X^n \ | \ {\bf x}_i\in D, \ \forall\ i\in S
\}.$$ Let $\mathcal{A}$ be the collection of $D_S$ for all
$S\subset N:=\{1,...,n\}$ with $|S|\ge 2$. It is clear that the
collection is a simple arrangement of smooth subvarieties of $X^n$
and take a building set $\cG =\mathcal{A}$. Then define
$X_D^{[n]}$  to be the closure of $X^n\setminus\bigcup _{S}D_S$ in
$$X^n\times \prod _{S} \mathrm{Bl}_{D_S}X^n$$

It can be constructed by a successive blowups by theorem
\ref{thm:existance}. In particular we may order $\cG$ as
$D_{12},D_{123};D_{13},D_{23}$;...; $D_{12...,n}$;
$D_{U\cup\{n\}}$ with $|U|=n-2$ and $U\subset N\setminus
\{n\}$;...; $D_{in}$ for $i=1,...,n-1$ by theorem \ref{thm:order}.

\begin{Lemma}\label{poly} Let $I_1$ and $J_2$ be partitions of $N$.
The intersection of proper transforms of $\Delta _{I_1}$ and $
\Delta _{I_2}$ is the proper transform of the intersection $\Delta
_{I_1\wedge I_2}$.
\end{Lemma}

%\begin{proof} Induction on $n$. As in \cite{FM}, we use $I$ for
 %a subset of $N$, and $I^+$ for $I \cup \{ n+1 \}$.  First consider
 %diagonals. Suppose that the proposition is true for $n$. Now we
%start with $Y_0:=X_D^{[n]}\times X$. Notice that $\Delta _I (Y_0)
%:= \Delta _I(X_D^{[n]})\times X$. Furthermore notice that, for $|I| >1$, $\Delta_{I^+}(Y_0)$ is the graph of $\pi _i : \Delta _I(X_D^{[n]}) \ra X$
%for any $i\in I$  and $\Delta _{\{ k \} ^+} (Y_0)$ is the graph of
%$\pi _k: X_D^{[n]}\ra X$.

%First by Lemma \ref{blowupLemma}, $\Delta _I (X_D^{[n+1]})=\pi
%^{-1}(\Delta _I (Y_0))$. For $\Delta _{I^+}(X_D^{[n+1]})$,
%consider $\Delta _{a^+}(X_D^{[n+1]})$. It can be considered as a
%section $\sigma _a$ of $\pi: X_D^{[n+1]}\ra X_D^{[n]}$ because of
%lemma \ref{blowupLemma}, (\ref{iso1}). It passes through smooth
%points of $\pi$. Notice that $\Delta _{I^+} \subset \Delta _{a^+}$
%if $a\in I$. Then to find local defining equations for $\Delta
%_{I^+}( X_D^{[n]})$ we may use analytical local box such that $\pi
%(x,z)=x$. On such local box, if $a\in I$,
%\[ \Delta _{I^+} (X_D^{[n+1]}) = \{ (x,z) \ | \ z =\sigma _a (x), x\in \Delta _I (X_D^{[n]}) \} \]
%(where $\Delta _I(X_D^{[n]}):=X_D^{[n]}$ if $|I|=1$).

%Ditto for polydiagonals $\Delta _{I_1,...,I_l}:=\bigcap
%_{i=1,...,l}\Delta _{I_i}$ where $I_i$ is a disjoint union of
%$N\cup\{n+1\}$, with $|I_i|\ge 2$. Then if $I_l$ contains both $a$
%and $n+1$, then
%\[\Delta _{I_1,...,I_l}  (X_D^{[n+1]})=
%\{ (x,z) \ | \ z = \sigma _a (x), x\in \Delta
%_{I_1,...,I_l\setminus \{ n+1 \}} (X_D^{[n]}) \} \]

%\end{proof}

\begin{Cor} For $I\subset N$ with $|I|\ge 2$,
$\Delta _I(X_D^{[n]})$ form a building set of nonsingular
subvarieties of $X_D^{[n]}$ with respect to the set of all
polydiagonals.
\end{Cor}

\begin{Def}
Define $X_D[n]$ to be the closure of $X_D^{[n]}\setminus \bigcup
_{|I|\ge 2} \Delta _I (X_D^{[n]})$
\[ X_D^{[n]}\times \prod _{|I|\ge 2} \mathrm{Bl}_{\Delta _I (X_D^{[n]})} X_D^{[n]}\]
\end{Def}

\begin{Thm}\label{properties}

\begin{enumerate}

\item $X_D[n]$ is a nonsingular variety. There is a natural
projection from $X_D[N]$ to $X_D[I]$ for any subset $I$ of $N$.
There is a natural $S_n$-action on $X_D[n]$.

\item The boundary is the union of divisors
$\widetilde{D_S}$ with $|S|\ge 1$, and $\widetilde{\Delta _I}$
with $|I|\ge 2$ of normal crossings.

\item The intersections of boundary divisors are nonempty if and
only if they are nested.  Here $\{ D_{S_i}, \Delta _{I_j} \}$ is nested if each pair $S_i$ and $S_k$
($T_j$ and $T_l$) is either disjoint or one is contained in the
other and each pair $S_i$ and $T_k$ is either disjoint or $T_k$ is
contained in $S_i$.

\item We may take order $D_S$; $\Delta _I$ for $n \not \in S, I$; and
then $D_T$ with $n\in T$, then $\Delta _J$ with $n\in J$.

\end{enumerate}

\end{Thm}

%\begin{proof} (1) and (2) follow from Theorem \ref{thm:existance}.
%It remains to prove only (3) and (4).

%(3): i) Suppose that both $S \cap I$ and $I\setminus S$ are
%nonempty. Then want to prove to $\widetilde{D_S}\cap
%\widetilde{\Delta _I}$ is empty. This is because $D_{S\cup I}$ is
%blown up already. Note that $D_S\cap \Delta _I \subset D_{S\cup
%I}$ and use Lemma \ref{blowupLemma}, (\ref{disjoint}).

%ii) Let  $\{ D_{S_i}, \Delta _{I_j} \}$
%be a nested set. Consider $$V:=\bigcap  D_{S_i}
%(X^{[n]}_D)$$ then note that $ V\cap \Delta _{I_j} (X^{[n]}_D)$ is
%nonempty. The set \[ \mathcal{G}_V:= \{V\cap \Delta _{I_j}
%(X^{[n]}_D) \ | \ i\in I_2 \} \] is a building set of $V$. Now it
%is easy to check that $V(X_D[n])$ coincides with the
%compactification $V_{\mathcal{G}_V}$ of $V$ with respect to the
%building set $\mathcal{G}_V$.

%iii) Disjointness of the pair $S_i$ and $S_k$ ($T_j$ and $T_l$)
%follows from theorem \ref{thm:existance}.

 %(4): Since $\widetilde{D_T}$ in $X_D^{[n-1]}\times X$ is
%transversal to $\Delta _I (X_D^{[n-1]}\times X)$ if $I\subset
%N\setminus\{n\}$, we may exchange the order of them.
%\end{proof}

\section{Chow groups and motives}
\label{groups-motives}
In this section, we will apply theorem \ref{ChowY} to $X_D^{ [n]}$
and $X_D [n]$.  For simplicity, we assume that $X$ is complete.

\subsection{Chow group and motive of $X_D^{ [n]}$}

In this case, our $Y= X^n, \cS=\cG= \{ D_S : S \subset N
\text{with} |S| \geqslant 2 \}$ where $D_S= \{{\bf x}\in X^n \ | \
{\bf x}_i\in D, \ \forall\ i\in S \}$.  We have $\cS=\cG$, so a
$\cG$-nest is just a chain of elements in $\cS$, $\cT = \{ D_{S_1}
\subset D_{S_2} \subset \cdots \subset D_{S_k} \}$.  Thus $Y_0 \cT
= D_{S_1}$.

A chain $\cCH$ is a chain of subset of $N$, $S_{k} \subsetneq
\cdots \subsetneq S_2 \subsetneq S_1$, such that $S_k$ is not a
singleton.  Obviously, there is one-to one correspondence between a set of chains of $\cS$ and a set of chains of $N$.  We say $\emptyset$ is also a chain.  We define
$\text{max}_{\cCH(\cT)} S$ as the maximal element of $\cCH(\cT)$ which is
strictly contained in $S$, where $\cCH(\cT)$ is the chain of $N$ which corresponds to $\cT$.  If there is no such element, then we
define $\text{max}_{\cCH(\cT)} S = \emptyset$

Now let $G= D_S$ and let's compute $r_{\cT}(G)$;

\begin{equation*}
\begin{split}
r_{\cT}(G) &= \text{dim} (\bigcap_{G \subsetneq T \in \cT} T) -
\text{dim} G \\
           &= \text{dim}(D_{\text{max}_{\cCH(\cT)} S}) - \text{dim} D_S \\
           &= |S|-|\text{max}_{\cCH(\cT)} S|.\\
\end{split}
\end{equation*}

\begin{Rmk}[When $D$ is not a divisor]
When $D$ is not a divisor, then we also blow up $D_{ \{i \} }$. So
we will not exclude the case such that $S_k$ is a singleton for $
\{ S_{k} \subsetneq \cdots \subsetneq S_2 \subsetneq S_1 \}$.
$r_{\cT}(G)$ will be also changed, it will be multiplied by the
codimension of $D$ in $X$.

\end{Rmk}

For a chain $\cCH(\neq \emptyset )$ , define
$$M_{\cCH}:= \{ \overrightarrow{\mu}= \{ \mu_{S} \}_{S \in \cCH}  : 1 \leq
\mu_S \leq |S| - |\text{max}_{\cCH} S| -1 \}.$$ For
$\cCH=\emptyset$, define $ M_{\cCH}$ is consist of one
$\overrightarrow{\mu}$ with $\| \overrightarrow{\mu} \|=0$ and
$D_{\emptyset}=X^n$.

\begin{Thm}
Let $X$ be a complete nonsingular variety.  Then we have the Chow
group and motive decompositions

$$ A^*(X_D^{ [n]}) = \bigoplus_{\cCH} \bigoplus_{\overrightarrow{\mu} \in M_{\cCH} }
A^{*- \| \overrightarrow{\mu} \|} (D_{S_{\cCH}}),$$

$$ h(X_D^{ [n]}) = \bigoplus_{\cCH} \bigoplus_{\overrightarrow{\mu} \in M_{\cCH} } h (D_{S_{\cCH}})(\|
 \overrightarrow{\mu} \| ) ,$$
where $\cCH$ runs through all the
chains of $N$ and $S_{\cCH}$ is the maximal element in $\cCH$.

\end{Thm}

\subsection{Chow group and motif of $X_D [n]$}
We use the same notation as \cite{Li2}.

\begin{enumerate}
\item We call two subsets $I, J \subset N$ are overlapped if $I
\cap J$ is not a nonempty proper subset of both $I$ and $J$.  For
a set $ \cN$ of subsets of $N$, we call $I$ is compatible with
$\cN$, denoted by $I \sim \cN$, if $I $ does not overlap any
elements of $\cN$.

A nest $\cN$ is a set of subset of $N$ such that any pair $I \neq
 J \in \cN$ are not overlapped and contains all singletons.

 For a given nest $\cN$, define $\cN^{\circ}:= \cN \setminus
\{ \{1\}, \cdots , \{n\} \}$.

A nest $\cN$ naturally corresponds to a tree (which may not be
connected) with each node is labeled by an element of $\cN$. Let
$c(\cN)$ be the number of connected components of the forest which
corresponds to $\cN$.  Denote by $c_{I} (\cN)$ the number of
maximal elements of the set $\{ J \in \cN: J \subsetneq I \}$,
which is called the number of sons of the node $I$.

Let $\overline{\Delta_{\cN}} := \cap_{I \in \cN} \Delta_I(X_D^{
[n]})$ in this section.

\item For a nest $\cN ( \neq \{ \{ 1 \}, \cdots \{n \} \})$,
define
$$M_{\cN}:= \{ \overrightarrow{\mu} = \{ \mu_{I} \}_{I \in \cN} : 1
\leq \mu_I \leq m( c_I -1)-1 \}$$ where $m= \text{dim} X$.

For $\cN=\{ \{ 1 \}, \cdots \{n \} \}$, define $M_{\cN}= \{
\overrightarrow{\mu} \}$ with $\| \mu \| = 0 $.
\end{enumerate}

As in \cite{Li2}, we have

\begin{Prop}\label{n-point-chow}
We have the Chow group and motive decompositions

$$ A^*(X_D [n]) = \bigoplus_{\cN} \bigoplus_{\overrightarrow{\mu} \in M_{\cN} }
A^{*- \| \overrightarrow{\mu} \|} (\overline{{\Delta}_{\cN}}),$$

$$ h(X_D [n]) = \bigoplus_{\cN} \bigoplus_{\overrightarrow{\mu} \in M_{\cN} } h (\overline{\Delta_{\cN}} )(\|
 \overrightarrow{\mu} \| ) ,$$

 where $\cN$ runs through all the
nest of $N$

\end{Prop}

Now we need to simplify $A^*(\overline{\Delta}_{\cN})$ and $h (\overline{\Delta}_{\cN})$.

\begin{Lemma}\label{diffclean}
$D_S$ and $\Delta_I$ intersect cleanly.

\end{Lemma}

\begin{proof}
We only need to prove that $T D_S \cap T \Delta_I \subset T (D_S \cap \Delta_I)$.  An arc in $\Delta_I$ have a coordinate representative $( {\bf{x}}_i) \in X^n$ such that $\bf{x}_i=\bf{x}_j$ for $i, j \in I$. For an arc in $\Delta_I$ to be an arc in $D_S$, ${\bf{x}}_i \in D$ for all $i \in S$.  Thus the arc should be an arc in $D_S \cap \Delta_I$.
\end{proof}

%\begin{Lemma}\label{prodblow}
%Let $V_1= (X \times \cdots \times X) \times \Delta, V_2 = D_{S'} \times (D \tim%es \cdots \times D)$ and $Z= D_{S''} \times (X \times \cdots \times X)$ where $%S' \subset S''$.  Consider the blow up $X^a \times X^b$ along $Z$.  Then $\wide%tilde{V_1}$ and $\widetilde{V_2}$ intersects cleanly. 
%\end{Lemma}

%\begin{proof}
%Except on the point over the blow up center $Z$, $\widetilde{V_1}$ and $\wideti%lde{V_1}$ intersect cleanly, so we need to check cleaness on the exceptional lo%cus.  We want to show that $T_p \widetilde{V_1} \cap T_p \widetilde{V_2} \subse%t T_p( \widetilde{V_1} \cap \widetilde{V_2})$ for $p \in \widetilde{V_1} \cap \%widetilde{V_2} \cap E$ where $E$ is the exceptional divisor.

%Let $l_t$ be an arc in $\widetilde{V_1} \cap \widetilde{V_2}$ which go through $%p$ at $t=0$.      

%\end{proof}

\begin{Prop}\label{deltaiso}
$\overline{\Delta_I}$ is isomorphic to $X^{[|I^c|+1]}_D$. 

\end{Prop}

\begin{proof}
We need to know which blow ups of $D_S$ have an effect to $\Delta_I$ in a specific order of blow ups.
We can assume that $I= \{ l, \cdots, n \}$ by arranging the order and denote $a= |I^c|$ and $b= |I|$.  We will denote $\Delta_I$ by $ X^a \times \Delta ( \cong X^{|I^c|+1})$. 
Then we have two different kinds of $D_S$.  The first one is that $S \subset I^c$, which we call  the first kind, the second one is that $S \nsubseteq I^c$, which we call the second kind.  We will change the order of blow ups so that we first blow up along $D_S$ of the first kind, and then along the second kind.  More precisely, we order $D_{I^c} \times X^b, D_{1, \cdots, \hat{i}, \cdots, l} \times X^b, \cdots , D_{i,j} \times X^b (i, j \in \{ 1, \cdots, a \})$ and then $D_{I^c} \times D^b, \cdots, D_{S'} \times D_{S''}, \cdots  (|S''|>0 \ \text{and} \ (|S'|, |S''|): \  \text{non-increasing in lexicographical order } )$.  
This order satisfies $(*)$-condition in definition/theorem \ref{thm:one-blow}, so that we can blow up in this order.
 
In this order of blow ups, notice that $\widetilde{X^{a} \times \Delta }$ and $\widetilde{D_{S'} \times D_{S''}}$ for $S'' \subsetneq I$ are separated when we blow up  along $\widetilde{D_{S'} \times D^b}$.  Thus we can forget the process of blow ups by $\widetilde{D_{S'} \times D_{S''}}$ where $S'' \subsetneq I $  i.e. we only need to care about $D_{S'} \times D^b$ for the second kind.   Under the isomorphism $X^a \times \Delta \cong X^{a+1}$, they are just $D_{S'} \times D$.

\end{proof}

We can also apply the same technique to polydiagonals term by term.  %For example, if we have a polydiagonal $X^a \times \Delta_I \times \Delta_J$, consider a sequence 
%$D_{R} \times X^{|I|} \times X^{|J|}$; $D_{S} \times D_{S'} \times X^{|J|}$ where $|S'| > 0$;  $D_{T} \times D_{T'} \times D_{T''}$ where $|T''| >0$; in appropriate non-increasing lexicographical order.  After blowing up along all $D_{S} \times D_{S'} \times X^{|J|}$, use the above theorem to eliminate one of $\Delta$ and use the theorem again. 

 Thus we can go further from proposition \ref{n-point-chow}.

\begin{Thm}
We have the Chow group and motive decompositions
$$ A^*(X_D [n]) = \bigoplus_{\cN} \bigoplus_{\overrightarrow{\mu} \in M_{\cN} }
(\bigoplus_{\cCH} \bigoplus_{\overrightarrow{\lambda} \in M_{\cCH} }
A^{*- \| \overrightarrow{\mu} \|- \| \overrightarrow{\lambda} \| } (D_{S_{\cCH}})),$$
$$ h(X_D [n]) = \bigoplus_{\cN} \bigoplus_{\overrightarrow{\mu} \in M_{\cN} }
(\bigoplus_{\cCH} \bigoplus_{\overrightarrow{\lambda} \in M_{\cCH} }
h(D_{S_{\cCH}}){( \| \overrightarrow{\mu} \|+ \| \overrightarrow{\lambda} \|) }) ,$$
where $\cN$ runs through all the nest of $\cN$ and $\cCH$ runs through all the chains of $c(\cN)$.

\end{Thm} 

\section{Chow rings}
\label{rings}
%{\em The generalization of Ulyanov's compactification}
In this section we assume that $X$ has a cellular decomposition
and $D$ is a smooth divisor of $X$ such that $D \hookrightarrow X$
is a Lefshetz embedding.  The reason we assume these conditions is that we need a Kunneth decomposition and S. Keel's formula for intersection ring of blow-up.

%\begin{Rmk}
%We have to assume that $D$ is a union of cells if we want to say the resulting space is a HI space.

%\end{Rmk}

\subsection{Chow ring of $X_D^{ [n]}$}

Note that $D_S \hookrightarrow D_{S'}$
for $S \supset S'$ and $D_S \hookrightarrow X^n$ are Lefshetz
embedding.

Obviously, the arrangement $\cA$ is regular, so we can apply
theorem \ref{chowring}.

\begin{Thm}\label{ringD}
The Chow ring $A^*(X_D^{ [n]})$ is isomorphic to the polynomial
ring $ A^*(X^n) [x_S]$ modulo the ideal generated by

\begin{enumerate}

\item $x_{S} \cdot x_{T}$ for  $ S, T $ that overlap,

\item $\cJ_{D_{S}/X^n} \cdot x_{S}$ for all $S$,

\item $P_{D_{S}/X^n} (-\Sigma_{S' \supset S} x_{S'})$ for
all $S$.
\end{enumerate}

\end{Thm}

\subsection{Chow ring of $X_D [n]$}
We will compute the Chow ring of $X_D [n]$ from $X_D^{[n]}$ by a sequence of blow ups along, which is same as \cite{FM},

$$\Delta_{\{ 1,2 \}}, \Delta_{ \{ 1,2,3 \}}, \Delta_{ \{1,3 \}}, \Delta_{ \{2,3\}}, \cdots, \Delta_{ \{ 1, \cdots, n \} }, \cdots, \Delta_{ \{1, n\}}, \cdots, \Delta_{ \{n-1, n\} }. $$

Let 

$$Y^{[i]}_i \rightarrow \cdots \rightarrow Y^{[i]}_{k+1} \rightarrow Y^{[i]}_k \rightarrow \cdots \rightarrow Y^{[i]}_0$$

be a part of the above sequence of blow-ups along  $$\Delta_{ \{1, \cdots, i+1 \}}, \cdots \Delta_{ \{ 1, \cdots, i-k-1, i +1\}}, \cdots \Delta_{ \{k,\cdots  , i ,i+1\}}, \cdots, \Delta_{\{1,i+1\}}, \cdots, \Delta_{\{i, i+1\}}.$$  
Note $1 \leq i \leq n-1$.

We will compute Chow rings of $Y^{[i]}_k$'s inductively by using theorem \ref{blowchow}.

\begin{Lemma}\label{disnon}
If $I'$ and $J'$ are subsets of $\{ 1, \cdots, i, i+1\}$ that overlap, then $\widetilde{\Delta_{I'}}$ and $\widetilde{\Delta_{J'}}$ are disjoint at $Y^{[i]}_k$, except, up to the order of $I'$ and $J'$, in exactly the following cases:
\begin{enumerate}
\item $I' =I \subset \{ 1, \cdots, i \}, |I| \leq i-k,  J'=J^+$, with $J \subset I$,

\item $I' = I^+, J' = J^+,$ with $I \cap J = \emptyset , |I\cup J| \leq i-k$
\end{enumerate}
\end{Lemma}

\begin{proof}
We change the order of blow ups in the following way;

$$D_S, \Delta_I; D_{S^+}, \Delta_{I^+}; D_{S^{++}},$$
where $S, I \subset \{1, \cdots, i\}, |I^+| \leq i-k+2$ and $S^{++} \nsubseteq  \{ 1, \cdots, i, i+1\}$.
After blowing up along $D_S, \Delta_I$, the space is $X_D[i] \times X^{(n-i)}$. If $I', J' \subset \{ 1, \cdots, i\}$, then $\widetilde{\Delta_{I'}}$ and $\widetilde{\Delta_{J'}}$ are disjoint by theorem \ref{properties}.

For $I'=I^+, J'=J^+$, $\widetilde{\Delta_{I^+}}$ is a product of the graph of $p_a: \overline{\Delta_{I}} \rightarrow X$ and $X^{n-i-1}$ where $a \in I$ and we use a convention $\Delta_{a} = X^{n}$.  Same for $\widetilde{\Delta_{J^+}}$ .  To have non-empty intersection, $I$ and $J$ must be nested by theorem \ref{properties}.  But we have an assumption that $I^+$ and $J^+$ overlap, so that $I$ and $J$ must be disjoint. $\widetilde{\Delta_{I^+}}$ and $\widetilde{\Delta_{J^+}}$ will be separted after blowing up along $\widetilde{\Delta_{(I \cup J)^+}}$.

Now let's move to the case that $I'=I \subset \{ 1, \cdots, i \}$ and $J' =J^+$.  In this case, $\widetilde{\Delta_I} = \overline{\Delta_I} \times X^{n-i}$.  To have non-empty intersectoin, $I$ and $J $ are nested, i.e. $J \subset I$ or $I \subset J$.  But the latter case $I' \subset J'$, which contradict to the assumption of overlapping.  Thus $J \subset I$. $\widetilde{\Delta_{I}}$ and $\widetilde{\Delta_{J^+}}$ will be separted after blowing up along $\widetilde{\Delta_{I^+}}$.

Note that $D_S$ and $\Delta_I$ are intersecting cleanly and its intersection is a proper subset of $\Delta_I$.
\end{proof}

\begin{Lemma}\label{trans+}
For $a \in I \subset \{1, \cdots, i\}$ such that $ 2 \leq |I| \leq i-k$, then at $Y^{[i]}_k$,
$$\widetilde{\Delta_{I^+}}= \widetilde{\Delta_I} \cap \widetilde{\Delta_{a^+}}.$$
\end{Lemma}

\begin{proof}

Proof is very similar to proposition \ref{poly}.

\end{proof}

\begin{Lemma}\label{blowdiv}
If $\widetilde{\Delta_{I'}}$ is a divisor in $Y^{[i]}_k$, then the inverse image $\pi^*(\widetilde{\Delta_{I'}})$ in  $Y^{[i]}_k$ is the divisor $\widetilde{\Delta_{I'}}$, except cases such that $I'=J \subset \{1, \cdots, i\}$ with $|J| =i-k$ and in that case
$$ \pi^*(\widetilde{\Delta_{J}}) =\widetilde{\Delta_J} +\widetilde{\Delta_{J^+}}.$$
\end{Lemma}

\begin{proof}
For the case described in the statement, by lemma \ref{trans+}, the statement is true.
For other cases, it is obvious that the disisor $\widetilde{\Delta_{I'}}$ does not contain any blow up center by considering the space $X_D^{[i]} \times X^{(n-i)}$.  
\end{proof}

For $a \in N$, let $p_a$ be the corresponding projection from $X^{n}$ to $X$, and for $a, b \in N$ (distinct), let $p_{a,b}$ be the projection from $X^{n}$ to $X^{a,b}$.  Let $[\Delta] \in A^{m}(X^{a,b})$ be the class of the diagonal, where $m = \text{dim} X$.  Define a polynomial $c_{a,b}(t) \in A^*(X^n)[t]$ be 
$$c_{a,b}(t) = \sum_{i=1}^m (-1)^i p_a^*(c_{m-i})t^i+ [ \Delta_{\{a,b\}}]$$
where $c_{m-i}$ is the $(m-i)$-th Chern class of $X$ and $[ \Delta_{\{a,b\}}] = p_{a,b}^* ( [\Delta]).$

Let's compute Chern polynomials and Lefshetz kernels of $\Delta$'s at the stage of $Y^{[1]}_0$.

\begin{Lemma}\label{0th}
\begin{enumerate}
\item\label{0lef} $\cJ_{ \Delta_{I}(Y^{[1]}_0)/ Y^{[1]}_0} = (\cJ_{\Delta_I / X^n}, x_S)$ where $S  \nsupseteq I$.
\item\label{0chern} $P_{ \Delta_{I}(Y^{[1]}_0)/ Y^{[1]}_0} (t) = P_{\Delta_I/ X^n}(t)$.
\end{enumerate}

\end{Lemma}

\begin{proof}
(\ref{0lef}) By the proof of proposition \ref{deltaiso}, we know that $\widetilde{D_S}$ for $S \nsupseteq I$ is disjoint from $\widetilde{\Delta_I}$, and others intersect cleanly and non-trivially.  By lemma \ref{lefker}, we have the statement.  

(\ref{0chern}) We know that $\widetilde{\Delta_I}$ is intersecting with $\widetilde{D_S}$ cleanly including the cases disjoint by the proof of proposition \ref{deltaiso}.  By lemma \ref{lefker}, we know that a Chern polynomial will not be changed.

\end{proof}

\begin{Prop}\label{kthchern}

\begin{enumerate}
\item For $a \in \{ 1, \cdots, i \}, 0 \leq k \leq i-1$, a Chern polynomial of $\widetilde{\Delta_{a^+}}$ at $Y^{[i]}_k$ is
$$c_{a, i+1}(-t +\sum_{ a, i+1 \in I'} D_k I').$$

\item For $ I \subset \{ 1, \cdots, i \}, 2 \leq |I| \leq i-k$, a Chern polynomial of $\widetilde{\Delta_{I^+}}$ at $Y^{[i]}_k$ is 
$$ (t+D_k I) \cdot c_{a, i+1}(-t+\sum_{I^+ \subset I'} D_kI')$$
for any $a \in I$.

\end{enumerate}
Here $D_k I$ is the divisor of $Y^{[i]}_k$ corresponding to $\Delta_I$.
\end{Prop}

\begin{proof}
Exactly same as \cite{FM}.
\end{proof}

\begin{Prop}\label{kthlef}
Let $I'=I^+ \subset \{1, \cdots, i, i+1 \}$ such that $|I'| = i-k+1$.  Then the restriction $\widetilde{ \Delta_{I'}} \rightarrow Y^{[i]}_k$ is Lefschetz embedding, and its Lefschetz kernel is generated by

\begin{enumerate}
\item $D_kJ'$ for any $J' \subset \{1, \cdots i, i+1 \}$ that overlaps with $I'$, except if $I \subset J' \subset \{1, \cdots i, i+1 \}$.

\item $\cJ_{\Delta_{I'}/ X^n}$.

\item $x_S$ for $S  \nsupseteq I'$. 
\end{enumerate} 
\end{Prop}

\begin{proof}
By lemma \ref{lefker}, $ \widetilde{\Delta_{I'}} \rightarrow Y^{[i]}_k$ is Lefschetz embedding.  

Now let's prove the statement for generators.  By lemma \ref{lefker}, we have to show that, for $J'$ which overlap  with $I'$, those exceptional cases are exactly blow up centers which intersect $\widetilde{\Delta_{I'}}$ with non-empty intersection.  The order of blow ups does not matter to the statement, so that we can change the order as we want.

First consider a case that $I' \cap J' \neq \emptyset$.  We can assume $i+1 \in I' \cap J'$ by changing numbering.  In this case, by lemma \ref{disnon}, we know exactly when the intersection is non-empty or not.

Now, consider a case that $I' \cap J' = \emptyset.$  We can assume that $J'=  \{1, \cdots, j\}$ and $I' \subset \{j+1, \cdots,j+ i+1\}$.   Then by the inductive construction of $X_D[n]$, it is obvious they intersect. %at the stage of $X_D[j] \times X^{n-j}$, $\widetilde{\Delta_{J'}}$ is a product of the graph of $\widetilde{\Delta_{J}} \rightarrow X$ and $X^{n-j-1}$ and $\widetilde{\Delta_{I'}}$ is a product of the graph of $X_D[j] \rightarrow X^{|I'|}$ and $X^{n-j-i-1}$.  Thus they intersect at the stage when $\widetilde{\Delta_{J'}}$ will be blew up.   

%Now, consider a case that $I' \cap J' = \emptyset.$  We can assume that $J'= J \cup {j+1}$ where $J \subset \{1, \cdots, j\}$ and $I' \subset \{j+2, \cdots, i+1\}$.   Then at the stage of $X_D[j] \times X^{n-j}$, $\widetilde{\Delta_{J'}}$ is a product of the graph of $\widetilde{\Delta_{J}} \rightarrow X$ and $X^{n-j-1}$ and $\widetilde{\Delta_{I'}}$ is a product of the graph of $X_D[j] \rightarrow X^{|I'|}$ and $X^{n-j-i-1}$.  Thus they intersect at the stage when $\widetilde{\Delta_{J'}}$ will be blew up.   

\end{proof}

\begin{Prop}\label{ChowYk}
For $0 \leq k \leq i$, $A^*(Y^{[i]}_k)$ is the polynomial ring $A^*(X^n)[x_S, D_k I]$, where $S \subset N$ such that $|S| >1$ and $I \subset \{1, \cdots, i+1 \}$ such that either $I \subset   \{1, \cdots, i \}$ or $|I| >i-k+1$, modulo the ideal generated by

\begin{enumerate}
\item\label{1} $D_k {I} \cdot D_k {J}$ for $I$ and $J$ that overlap,

\item\label{2} $x_{S} \cdot D_k {I}$ unless $I \subset S$,

\item\label{3} $\cJ_{\Delta_{I}/X^n} \cdot D_k {I}$ for all $I$,

\item\label{4} 
\begin{enumerate}
\item $c_{a,b} ( \sum_{a, b \in I} D_k I)$ for $a, b \in \{1, \cdots, i \}$ (distinct);

\item\label{i+1con} $D_k I \cdot c_{a, i+1}( \sum_{I^{+}  \subset I'} D_k {I'} )$ for $I \subset \{1, \cdots, i \}, |I| >i-k, a \in I$ and $I^{+} = I \cup \{i+1\}$
\end{enumerate}

\item\label{5} $x_{S} \cdot x_{T}$ for $S$ and $T$ that overlap,

\item\label{6} $\cJ_{D_{S}/X^n} \cdot x_{S}$ for all $S$,

\item\label{7} $P_{D_{S}/X^n} (-\Sigma_{S' \supset S} x_{S'})$ for
all $S$.

\end{enumerate}

\end{Prop}

\begin{proof}

For $Y^{[1]}_0$, it is just theorem \ref{ringD}.  
Also note that $Y^{[i]}_i = Y^{[i+1]}_0$ and the statement for $Y^{[i]}_i$ will imply $Y^{[i+1]}_0$ because condition (\ref{i+1con}) is vacuous when $k=0$.

We only need to prove that the statement for $Y^{[i]}_k $ will imply the one for$Y^{[i]}_{k+1} $.  The conditions (\ref{5}) to (\ref{7}) are coming from blow up along $D_S$ and these are not new.

For (\ref{4}), proof is exactly same as \cite{FM}.

(\ref{1}), (\ref{2}), and (\ref{3}) follow from proposition \ref{kthlef}.

\end{proof}

Especially, we have

\begin{Thm}
The Chow ring $A^*(X_D [n])$ is isomorphic to the polynomial
ring $ A^*(X^n) [x_S, y_I] $ modulo the ideal generated by

\begin{enumerate}
\item $y_{I} \cdot y_{J}$ for $I$ and $J$ that overlap,

\item $x_{S} \cdot x_{T}$ for $S$ and $T$ that overlap,

\item $x_{S} \cdot y_{I}$ unless $I \subset S$,

\item $\cJ_{\Delta_{I}/X^n} \cdot y_{I}$ for all $I$,

\item $\cJ_{D_{S}/X^n} \cdot x_{S}$ for all $S$,

\item $c_{a,b} ( \sum_{a, b \in I} y_I)$ for $a, b \in \{1, \cdots, n \}$ (distinct),

\item $P_{D_{S}/X^n} (-\Sigma_{S' \supset S} x_{S'})$ for
all $S$.

\end{enumerate}

\end{Thm}

%\bibliography{math}

%%%%%%%%%%%%%%%%%%%%%%%%%%%%%%%%%%%%%%%%%%%%%%%%%%%%%%%%%%%%%%%%%%%%%%%%%%%

%%%%%%%%%%%%%%%%%%%%%%%%%%%%%%%%%%%%%%%%%%%%%%%%%%%%%%%%%%%%%%%%%%%%%%%%%%%
\def\cprime{$'$}
\providecommand{\bysame}{\leavevmode\hbox
to3em{\hrulefill}\thinspace}

\end{document}